\documentclass[10pt, a4paper]{amsart}
\usepackage[utf8]{inputenc}
\usepackage{hyperref}
\usepackage{amsfonts}
\usepackage{amssymb}
\usepackage{amsmath}
\usepackage{amsthm}
\usepackage{newlfont}
\usepackage[mathcal]{euscript}
\usepackage{tikz}
\usepackage{soul}
\usepackage{enumerate}
\usepackage{thm-restate}
\usetikzlibrary{shapes.geometric}
\usetikzlibrary{decorations.pathreplacing}

 \newtheorem{thm}{Theorem}

\newtheorem{lem}[thm]{Lemma}

\newtheorem{lemma}[thm]{Lemma}

\newtheorem{example}[thm]{Example}

\newtheorem{remark}[thm]{Remark}

\DeclareMathOperator{\Int}{int}
\DeclareMathOperator{\diam}{diam}

\newcommand{\degr}{\mathrm{deg}}

\newcommand{\dist}{\mathrm{dist}}

\newcommand{\di}{\mathrm{d}}

\newcommand{\crit}{\mathrm{Crit}}

 \newcommand{\eps}{\varepsilon}

 \def\B{{\mathcal B}}

 \newcommand{\A}{\mathcal{A}}

 \newcommand{\N}{\mathbb{N}}
 \newcommand{\Z}{\mathbb{Z}}

 \newcommand{\R}{\mathbb{R}}
 
 \newcommand{\Ci}{\mathbb{S}^1}

 \def\Fh{\hat F}

\begin{document}

\title[Are generic dynamical properties stable?]{Are generic dynamical properties stable under composition with rotations?}
\author{Jozef Bobok}
\author{Jernej \v Cin\v c}
\author{Piotr Oprocha}
\author{Serge Troubetzkoy}

\address[J.\ Bobok]{Department of Mathematics of FCE, Czech Technical University in Prague,
Th\'akurova 7, 166 29 Prague 6, Czech Republic}
\email{jozef.bobok@cvut.cz}

\address[J.\ \v{C}in\v{c}]{AGH University of Science and Technology, Faculty of Applied Mathematics,
al.\ Mickiewicza 30, 30-059 Krak\'ow, Poland. -- and -- Centre of Excellence IT4Innovations - Institute for Research and Applications of Fuzzy Modeling, University of Ostrava, 30. dubna 22, 701 03 Ostrava 1, Czech Republic}
\email{jernej.cinc@osu.cz}

\address[P.\ Oprocha]{AGH University of Science and Technology, Faculty of Applied Mathematics,
al.\ Mickiewicza 30, 30-059 Krak\'ow, Poland. -- and --
Centre of Excellence IT4Innovations - Institute for Research and Applications of Fuzzy Modeling, University of Ostrava, 30. dubna 22, 701 03 Ostrava 1, Czech Republic}
\email{oprocha@agh.edu.pl}

\address[S.\ Troubetzkoy]{Aix Marseille Univ, CNRS, I2M, Marseille, France
postal address: I2M, Luminy, Case 907, F-13288 Marseille Cedex 9, France}
\email{serge.troubetzkoy@univ-amu.fr}

\date{\today}


\begin{abstract}
In this paper we provide a detailed topological and measure-theoretic study of Lebesgue measure-preserving circle maps that are rotated with inner and outer rotations which are independent of each other. In particular, we analyze  the stability of the locally eventually onto and measure-theoretic mixing properties.
\end{abstract}

\maketitle

\section{Introduction}
 We study the structure of Lebesgue measure-preserving maps of the circle, these are one-dimensional versions of volume-preserving maps, or more broadly, conservative dynamical systems; ergodic maps preserving Lebesgue measure are the most fundamental examples of maps having a unique physical measure.
 In recent years there has been extensive interest in piecewise linear Lebesgue measure-preserving  one dimensional case, the most classical example being {\em interval exchange transformations (IETs)}. Ergodic properties of the Lebesgue measure for IETs are of prime interest. Keane  proved the minimality of most IETs
and showed that there is a large class of IETs which are not uniquely ergodic but Lebesgue measure is ergodic  \cite{Kea1},\cite{Kea2}, thus initiating questions about typical properties of IETs or more generally Lebsegue measure-preserving transformations. These questions have attracted much attention, for example  Masur \cite{Ma} and Veech \cite{Ve} proved that typical minimal IETs are uniquely ergodic. This line of research more recently culminated in the work of Avila and Forni \cite{AF}, who proved among other things that typical IETs are weakly mixing.
 
The question we study in this article is the following;
we consider a dynamical property and ask for which circle maps $f$  satisfying this property, when we compose $f$ with  circle rotations,  this property still holds for the resulting map $g := r_{\alpha} \circ f \circ r_{\beta}$.

Such a question was studied for example in
 \cite{AB}. The authors showed that for any interval exchange transformation $T$ the map $T \circ r_{\beta}$ is uniquely ergodic for almost every $\beta$.  This nice result motivated us to initiate our study.  

We place ourselves in the framework of continuous Lebesgue measure-preserving maps of the circle, denoted $C_{\lambda}(\mathbb{S}^1)$.
This setting is in some sense complementary to the setting of IETs:  except for rotations,  IETs are invertible piecewise monotone but not continuous while
maps in $C_{\lambda}(\mathbb{S}^1)$ are  not invertible and  not piecewise monotone but continuous.
We define an operator $T_{\alpha,\beta}: C_{\lambda}(\mathbb{S}^1) \mapsto  C_{\lambda}(\mathbb{S}^1)$ by $T_{\alpha,\beta}(f)=r_{\alpha}\circ f\circ r_{\beta}.$ A more precise version of the question we ask is:
\begin{quote}
 Given a dynamical property, how large is  the set of $f \in C_{\lambda}(\mathbb{S}^1)$ such that this property holds for {\em all} $T_{\alpha,\beta}(f)$.
 \end{quote}

This question indeed makes sense because $C_{\lambda}(\mathbb{S}^1)$ is closed under the operation $T_{\alpha,\beta}$. This is not true anymore in the setting of circle maps that preserve a non-atomic invariant measure $\mu\neq \lambda$ with full support.
It is also not true in the set of all continuous circle maps with dense set of periodic points, this result will be presented in a separate article which is under preparation.  

Furthermore, we note that in this paper we do not study the weaker version of the above question when {\em all} is replaced by {\em almost all}. 

We begin by studying topological dynamical properties. The strongest topological dynamical property in our setting is the {\em locally eventually onto (leo)} property (in the literature sometimes referred to by {\em topological exactness}).   
Our first result, which was surprising for us, is
\begin{restatable}[]{thm}{main}\label{thm:main}
There is an open dense set of maps $\mathcal{O}\subset C_{\lambda}(\Ci)$ such that:
\begin{enumerate}
\item\label{thm:main1} each $f\in \mathcal{O}$ is leo.
\item\label{thm:main2}for each pair $\alpha,\beta\in [0,1)$ and each $f\in \mathcal{O}$ the map $T_{\alpha,\beta}(f)\in \mathcal{O}$.
\end{enumerate}
\end{restatable}
According to the usual hierarchy of topological dynamical properties, every map in $\mathcal{O}$ is also topologically mixing, topologically weakly mixing, totally transitive and transitive.

Notice that this theorem is stronger than the above mentioned result for interval exchange transformations \cite{AB}  in the sense that it holds for all $\alpha,\beta$, but weaker in the sense that it does not hold for all maps $f \in C_{\lambda}(\mathbb{S}^1)$. We show that such a statement can not hold for all $f\in C_{\lambda}(\mathbb{S}^1)$;
in Section \ref{sec:examp}
we give an explicit example of a map $f \in C_{\lambda}(\mathbb{S}^1)$ for which $T_{\alpha,\beta}(f)$ is not transitive for an open set
of $\alpha,\beta$ (Example \ref{ex:tent}).
This example also shows that $C_{\lambda}(\mathbb{S}^1)\setminus\mathcal{O}$ must be nonempty even in the {\em almost all} version of the question.

Next we ask the same question in the measure-theoretic framework. In \cite{BCOT}  we showed that there is a dense set of  non-ergodic maps in $C_{\lambda}(\mathbb{S}^1)$, thus we can not have a nonempty open set of maps satisfying nice mixing properties for all $\alpha,\beta$.
None the less we obtain the following (optimal) result.

\begin{restatable}[]{thm}{tfive}\label{t:5''}
There is a dense $G_{\delta}$ subset $\mathcal{G}$ of $C_{\lambda}(\mathbb{S}^1)$ such that
\begin{enumerate} \item each $g \in \mathcal{G}$ is weakly mixing with respect to $\lambda$, and
\item
for each pair $\alpha,\beta\in [0,1)$ and each $f\in \mathcal{G}$ the map $T_{\alpha,\beta}(f)\in \mathcal{G}$.
\end{enumerate}
\end{restatable}

Another aim of our research is to understand natural conditions when topological dynamical properties imply the corresponding measure theoretical properties. There is a nice survey article by Glasner and Weiss with the following description of the interplay between measure theoretic and topological dynamics:
{\em The two sister branches of the theory of dynamical systems called ergodic theory (or
measurable dynamics) and topological dynamics $\dots$ describe different
but parallel notions in their respective theories and the surprising fact is that many of
the corresponding results are rather similar \cite{GW}.}
 The next two results are of this type in our particular context. The smoothness assumptions in the next two theorems come directly from the paper of Li and Yorke \cite{LY} and Bowen \cite{Bowen} respectively.

\begin{restatable}[]{thm}{ppp}
\label{p:1}Let $f\in C_{\lambda}(\Ci)$ be a piecewise $C^2$ map with slope strictly greater than $1$.
Then $f$ is transitive if and only if $(f,\Ci,\mathcal{B},\lambda)$ is ergodic.
\end{restatable}

It is well known that measure-theoretic exactness implies measure-theoretic strong mixing, and since $\lambda$ is positive on open sets this implies topological mixing.  The next result is a partial converse to this, and is an important ingredient of the proof of Theorem \ref{thm:main2}.

\begin{restatable}[]{thm}{exact}\label{exactness}
Let $f\in C_{\lambda}(\mathbb{S}^1)$ be a piecewise $C^2$ topological mixing map. Then $(f,\Ci,\mathcal{B},\lambda)$ is measure-theoretically exact (thus also strongly and weakly mixing). 
\end{restatable}

 From the proofs it is immediate that Theorems \ref{p:1} and  \ref{exactness} also hold in the setting of $C_{\lambda}([0,1])$, which we also could not find in the literature.

To prove Theorem \ref{thm:main1}  we prove several intermediary results of independent interest on special dense classes of maps in $C_{\lambda}(\mathbb{S}^1)$.  The following result is a corollary of the  technical Lemma \ref{lem:leo-exp}. Again notice that 
 if $f$ satisfies the assumptions of the  corollary  then $T_{\alpha,\beta}(f)$ verifies them for all $\alpha,\beta \in [0,1)$.

\begin{restatable}[]{cor}{corleo}\label{cor:leo}
If a piecewise affine Lebesgue measure-preserving circle map has different critical values and slope strictly greater than $4$ then it is leo.
\end{restatable}

Another important ingredient in the proof of Theorem \ref{thm:main} is the complete geometric description of {\em periodic arcs} for sufficiently expanding piecewise monotone maps in $C_{\lambda}(\mathbb{S}^1)$ (Theorem \ref{periodn}).

All our techniques are a priori of one-dimensional nature, but we wonder if similar results still hold in higher dimensions.
In particular, it would be interesting to know which dynamical properties hold generically or on open sets of continuous non-invertible volume preserving maps, and if they commute with rotations on higher dimensional tori.

Dynamical and topological properties of typical continuous Lebesgue measure-preserving maps in the one dimensional setting have been studied in
\cite{Bo,BT,BCOT,BCOT2,CO,SW}. In particular, in the setting of $C_{\lambda}(\Ci)$ we completely characterized the sets of periodic points of typical maps  \cite{BCOT} and showed in \cite{BCOT2} a surprising result that a strong version of shadowing property, called the s-limit shadowing property, is generic.  We have not investigated if these properties commute with rotations.

\section{Preliminaries}\label{sec:prelim}

\subsection{Notation and background material} Let $\N:=\{1,2,3,\ldots\}$ denote the set of {\em natural numbers} and $\Z:=\{\ldots,-1,0,1,\ldots\}$ the set of integers. We denote by $\mathbb{S}^1$ the {\em unit circle}. For sets $A,A'\in \mathbb{S}^1$ we use notation $\dist(A,A')$ to denote the {\em Euclidean distance on $\mathbb{S}^1$ between sets $A$ and $A'$}.
Throughout the article we will identify $\mathbb{S}^1$ with the interval $[0,1)$. Let $\lambda$ denote \emph{Lebesgue measure} on the unit interval $[0,1]$ and by the abuse of notation we also denote by $\lambda$ the \emph{normalized Lebesgue measure} on $\Ci$. We denote by $C_{\lambda}(\mathbb{S}^1)\subset C(\mathbb{S}^1)$ the family of all {\em continuous Lebesgue measure-preserving maps} of $\mathbb{S}^1$ being a proper subset of the family of all continuous circle maps equipped with the \emph{uniform metric} $\rho$:
$$\rho (f,g) := \sup_{x \in \mathbb{S}^1} |f(x) - g(x)|.$$

It follows from Proposition 4 from \cite{BCOT} that $(C_{\lambda}(\mathbb{S}^1),\rho)$ is a complete metric space (formally this theorem was stated for the interval and not the circle, but the proof in the circle case  is analogous).

By {\em arc} we mean any subset of $\Ci$ which is a  homeomorphic image of $[0,1]$.
 We denote by $\mathcal A$ the set of all subarcs of $\mathbb S^1$. 
 
 \subsection{Dynamical properties}
 We say that $f$ has a {\em periodic arc $A \in \mathcal A$ of {\it minimal} period $n\in\N$} if
\begin{equation}\label{e:1a}
 f^{i}(A)\neq f^{j}(A) \text{ for }0\le i<j<n \text{ and } f^n(A)=A.
 \end{equation}

A map $f\in C_{\lambda}(\mathbb S^1)$ is called
\begin{itemize}
\item \textit{transitive} if for each pair of nonempty open sets $U,V\subset \Ci$ there is $n\ge 0$ such that $f^n(U)\cap V\neq\emptyset$,
\item \textit{totally transitive} if $f^n$ is transitive for all integers $n\geq1$.
\item \textit{topologically mixing} if for each pair of nonempty open sets $U,V\subset \Ci$ there is $n_0\geq0$ such that $f^n(U)\cap V\neq\emptyset$ for every $n\ge
    n_0$,
\item  \textit{leo} (\textit{locally eventually onto}) if for every nonempty open set $U\subset \Ci$ there is $n\in{\mathbb N}$ such that $f^n(U)=\mathbb S^1$.
\end{itemize}
Let  $\mathcal{B}$ denote the Borel sets in $\mathbb{S}^1$.
The measure-preserving transformation  $(f,\mathbb{S}^1,\mathcal{B},\lambda)$ is
\begin{itemize}
\item \textit{ergodic} if for each $A\in \mathcal{B}$, $f^{-1}(A)=A$ $\lambda$-a.e. implies $\lambda(A)=0$ or $\lambda(A^c)=0$.
\item  \textit{weakly mixing}, if for every $A,B\in\B$,
$$\lim_{n\to\infty}\frac{1}{n}\sum_{j=0}^{n-1}\vert \lambda(f^{-j}(A)\cap B)-\lambda(A)\lambda(B)\vert=0.$$
\item {\em measure-theoretically exact} if for every set $A\in \cap_{n\geq 0}f^{-n}(\mathcal{B})$ it holds that $\lambda(A)\lambda(A^{c})=0$.
\end{itemize}

\subsection{Piecewise monotone maps}
We say that a piecewise monotone map $f\in C_{\lambda}(\Ci)$ has {\em slope} $> s$ if $\left | \frac{f(x)-f(y)}{x-y} \right |>s$ whenever $x\neq y$ are points on an arc of monotonicity of $f$.

A \emph{critical point} (or {\em turning point}) of $f\in C(\Ci)$ is a point $x\in \Ci$ such that $f(x)$ is a strict local extremum. Denote by $\crit(f)$ the set of all critical points of $f$.  Let $\mathrm{PA}(\mathbb{S}^1)\subset C(\mathbb{S}^1)$ denote the set of \emph{piecewise affine} 
maps; i.e., there exists a finite partition of $\mathbb{S}^1$ so that the map is affine on each element of the partition. Note that the endpoints of affine pieces are not necessarily critical points by our definition.
 Let $\mathrm{PA}_{\lambda}(\mathbb{S}^1)\subset C_{\lambda}(\mathbb{S}^1)$ denote the set of {\em piecewise affine maps that preserve Lebesgue measure}.

\subsection{Rotations and liftings} For $\alpha\in [0,1)$ we define the map $r_{\alpha}\colon~\Ci\to \Ci$ as
$$r_{\alpha}(x)=x+\alpha\pmod1,~x\in\Ci$$
and for every $(\alpha,\beta)\in [0,1)^2$ the map $T_{\alpha,\beta}\colon~C_{\lambda}(\Ci)\to C_{\lambda}(\Ci)$ by
$$T_{\alpha,\beta}(f)=r_{\alpha}\circ f\circ r_{\beta}.$$
When there is no confusion about the map that we use we simply write $T_{\alpha,\beta}=T_{\alpha,\beta}(f)$.

Note that for any $x\in \mathbb{S}^1$ and any $\alpha,\beta\in [0,1)$ we have 
$$|T_{\alpha,\beta}(f)(x)-T_{\alpha,\beta}(g)(x)|=|f(r_{\beta}(x))-g(r_{\beta}(x))|.$$
Thus $\rho(T_{\alpha,\beta}(f),T_{\alpha,\beta}(g))=\rho(f,g)$. Therefore, $T_{\alpha,\beta}$ is an isometry.

Consider a continuous map $f\colon~\mathbb S^1\to \mathbb S^1$ of
\emph{degree} $\degr(f)\in\Z$. Let $F\colon~\R\to\R$ be a \emph{lifting} of $f$, i.e., a continuous map
for which
\begin{equation}\label{e:10}
\phi\circ F = f\circ\phi\text{ on }\R,
\end{equation}
where $\phi\colon~\R\to\Ci$ is defined by $\phi(x) = x \pmod{1}$ 
and $F(x + 1) = F(x) + \degr(f)$ for each $x\in\R$.  Note that two liftings of $f$ differ by an integer constant.

 \subsection{Window perturbations} 
 
 Fix an arc $A \in \A$.
 Let $\{A_i \in \A: 1 \le i \le m \}$, where $m$ is odd, be a finite collection of arcs satisfying $\cup^{m}_{i=1} A_i = A$ and $\Int(A_i) \cap \Int(A_j) = \emptyset $ when $i \ne j$, we will refer to this as a {\em partition} of $A$.
 
 Fix $f \in C_{\lambda}(\mathbb{S}^1)$
 an arc $A \in \A$ and a partition of $A$.
 A map $h = h_{A,\{A_i\}}$ is {\em an $m$-fold window perturbation of $f$ with respect to $A$ and the partition $\{A_i\}$} if
 \begin{itemize}
     \item $h|_{A^c} = f|_{A^c}$ 
     \item for each $1 \le i \le m$ the map $h|_{A_i}$ is an affinely scaled copy of $f|_{A}$ with the orientation reversed for every second $i$, with  $h|_{A_1}$ having the same orientation as $f|_{A}$.
 \end{itemize}

The essence of this definition is illustrated by Figure~\ref{fig:perturb}.

We call an $m$-fold window perturbation of $f$ {\em regular} if all of the $A_i$'s have the same length.

\begin{figure}[!ht]
	\centering
	\begin{tikzpicture}[scale=3.5]
	\draw (0,0)--(0,1)--(1,1)--(1,0)--(0,0);
	\draw[thick] (0,1)--(1/2,0)--(1,1);
	\node at (1/2,1/2) {$f$};
	\node at (5/16,-0.1) {$a$};
	\node at (5/8,-0.1) {$b$};
	\node at (0,-0.1) {$0$};
	\node at (1,-0.1) {$1$};
	\node at (-0.1,1) {$1$};
	\draw[dashed] (5/16,0)--(5/16,3/8)--(5/8,3/8)--(5/8,0);
	\node[circle,fill, inner sep=1] at (5/16,3/8){};
	\node[circle,fill, inner sep=1] at (5/8,1/4){};
	\end{tikzpicture}
	\hspace{1cm}
	\begin{tikzpicture}[scale=3.5]
	\draw (0,0)--(0,1)--(1,1)--(1,0)--(0,0);
	\draw[thick] (0,1)--(5/16,3/8)--(5/16+3/48,0)--(5/16+5/48,1/4)--(5/16+7/48,0)--(5/16+10/48,3/8)--(5/16+13/48,0)--(5/8,1/4)--(1,1);
	\draw[dashed] (5/16,0)--(5/16,3/8);
	\draw[dashed] (5/16+10/48,0)--(5/16+10/48,3/8);
	\draw[dashed] (5/16+5/48,3/8)--(5/16+5/48,0);
	\node at (1/2,1/2) {$h$};
	\node at (0,-0.1) {$0$};
	\node at (1,-0.1) {$1$};
	\node at (-0.1,1) {$1$};
	\node at (5/16,-0.1) {$a$};
	\node at (5/8,-0.1) {$b$};
	\draw[dashed] (5/16,3/8)--(5/8,3/8)--(5/8,0);
	\node[circle,fill, inner sep=1] at (5/16,3/8){};
	\node[circle,fill, inner sep=1] at (5/16+5/48,1/4){};
	\node[circle,fill, inner sep=1] at (5/16+10/48,3/8){};
	\node[circle,fill, inner sep=1] at (5/8,1/4){};
	\end{tikzpicture}
	\caption{For  $f\in C_{\lambda}(\mathbb{S}^1)$ shown on the left, we show on the right the graph of $h$ which is
 the regular $3$-fold piecewise affine window perturbation of $f$ on the interval $[a,b]$.}\label{fig:perturb} 
\end{figure}
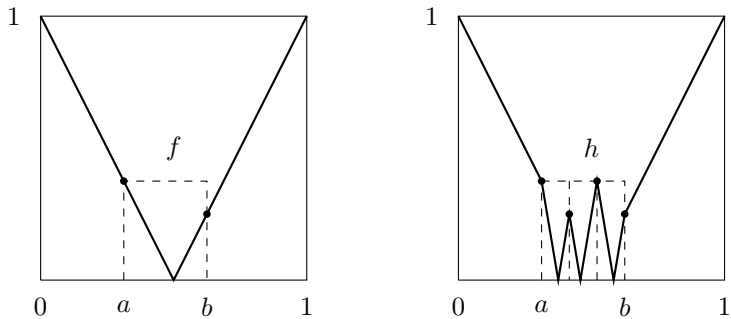

We collect some facts about window perturbations which were proven or discussed in \cite{BT},\cite{BCOT}.
Suppose $A = [a,b]$, and $h = h_{[a,b]}$ is an $m$-fold window perturbation of $f \in C_{\lambda}(\mathbb{S}^1)$ with respect to an arbitrary partition
$\{A_i\}$  of $A$. Then

\begin{enumerate}
    \item $h$ preserves $\lambda$.
    \item if $m$ is odd then $h$ is continuous,
    and thus $h\in C_{\lambda}(\mathbb{S}^1)$.
    \item  $\rho (h|_{[a,b]}, f) \to 0$ when  $|b-a| \to 0$.
    \item if $f\in \mathrm{PA}_{\lambda}(\mathbb{S}^1)$ then $h\in \mathrm{PA}_{\lambda}(\mathbb{S}^1)$. 
\end{enumerate}

\section{Motivational examples}\label{sec:examp}
In this section we give several examples that initially motivated our study.

\begin{example}\label{ex:tent}
Consider the full tent map $f$ (viewed as a circle map). Obviously, $f\in C_{\lambda}(\mathbb{S}^1)$. For $\alpha< -\beta \pmod1$ and $\alpha+\beta> -\frac{1}{2} \pmod1$ we define the interval $J=J_{\alpha,\beta}:=[-2\beta-\alpha,1+\alpha] \pmod1$. We obtain that $T_{\alpha,\beta}(J)=J$. Therefore, $T_{\alpha,\beta}$ is not transitive for an open subset of $(\alpha,\beta)\in[0,1)^2$.

\end{example}

\begin{figure}[!ht]
	\centering
	\begin{tikzpicture}[scale=3.5]
	\draw (0,0)--(0,1)--(1,1)--(1,0)--(0,0);
	\draw[thick] (0,0)--(1/2,1)--(1,0);
	\node at (1/2,1/4) {$f$};
	\node at (7/32,-0.1) {\small $-\alpha-\beta$};
	\node at (24/32,-0.1) {\small $1+\alpha+\beta$};
	\node at (1/2,-0.13) {};
	\draw[dotted] (0,0)--(1,1);
	\draw[dashed] (1/4-1/32,0)--(1/4-1/32,1);
    \draw[dashed] (3/4+1/32,0)--(3/4+1/32,1);
    \draw[dashed] (0,1/2-1/16)--(1,1/2-1/16);
    \node[circle,fill, inner sep=1] at (1/4-1/32,0){};
	\node[circle,fill, inner sep=1] at (3/4+1/32,0){};
	\node[circle,fill, inner sep=1] at (1,1/2-1/16){};
	\node[circle,fill, inner sep=1] at (1,1){};
	\draw[thick] (1/4-1/32,0)--(3/4+1/32,0);
	\draw[thick] (1,1/2-1/16)--(1,1);
	\node at (0,-0.1) {$0$};
	\node at (1,-0.1) {$1$};
	\node at (-0.1,1) {$1$};
	\end{tikzpicture}
	\hspace{1cm}
	\begin{tikzpicture}[scale=3.5]
	\draw (0,0)--(0,1)--(1,1)--(1,0)--(0,0);
	\draw[thick] (1/8+1/32,0)--(19/32,3/4+1/8)--(1,1/16);
	\draw[dotted] (0,0)--(1,1);
	\node at (1/4+1/16,-0.1) {\small $-2\beta-\alpha$};
	\node at (1/2+1/12,-0.1) {$J$};
	\node at (3/4+1/16,-0.1) {\small $1+\alpha$};
	\draw[dashed] (1/4+1/16,0)--(1/4+1/16,1);
	\draw[dashed] (0,1/16)--(1,1/16);
    \draw[dashed] (3/4+1/8,0)--(3/4+1/8,1);
    \draw[dashed] (0,1/4+1/16)--(1,1/4+1/16);
    \draw[dashed] (0,3/4+1/8)--(1,3/4+1/8);
    \draw[thick] (1/4+1/16,0)--(3/4+1/8,0);
    \draw[thick] (1,1/4+1/16)--(1,3/4+1/8);
    \node at (1.3,1/2+1/16) {$T_{\alpha,\beta}(J)=J$};
    \node at (1.2,28/32) {$1+\alpha$};
     \node[circle,fill, inner sep=1] at (1/4+1/16,0){};
	\node[circle,fill, inner sep=1] at (3/4+1/8,0){};
	\node[circle,fill, inner sep=1] at (1,1/4+1/16){};
	\node[circle,fill, inner sep=1] at (1,3/4+1/8){};
	\node at (0,-0.1) {$0$};
	\node at (1,-0.1) {$1$};
	\node at (-0.1,1) {$1$};
	\draw[dashed] (1/32,0)--(1/32,1);
	\draw[dashed] (5/32,0)--(5/32,1);
	\draw[thick] (0,1/16)--(1/32,0);
	\draw[thick] (5/32,1)--(3/32,14/16)--(1/32,1);
	\end{tikzpicture}
	\caption{ $\alpha=-1/8$ and $\beta=-3/32$ from Example~\ref{ex:tent}.}\label{fig:perturbations}
\end{figure}
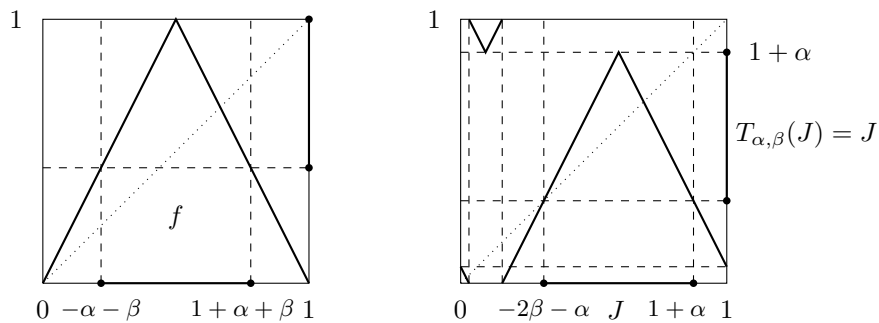

 On the other hand, the following simple example shows that we can expect maps from $C_{\lambda}(\mathbb{S}^1)$ to be leo and measure-theoretically exact wrt $\lambda$ provided they have slopes big enough even if they have the same critical values.

\begin{example}\label{ex:1}
For the map $g\in C_{\lambda}(\mathbb{S}^1)$ given on Figure~\ref{fig:ex1} the map $T_{\alpha,\beta}(g)$ is leo for every $\alpha,\beta\in [0,1)$. \\
Rotation does not change the slope, so the slope of $T_{\alpha,\beta}(g)$ stays $5$. Fix any arc $A\subset \mathbb{S}^1$, let $G$ be a lifting of $T_{\alpha,\beta}(g)$ and let $J\subset I$ be a lifting of $A$.
Note that if $J$ does not contain at least three critical points of $G$ then $\diam G(J)\geq \frac{5}{3} \diam J$. But then there is $n$
such that $G^n(J)$ contains three critical points, since distance between critical points is uniformly bounded. This means that $T_{\alpha,\beta}(g)^n(A)=\mathbb{S}^1$.\\
Applying Theorem~\ref{exactness} shows that the map $T_{\alpha,\beta}(g)$ is measure-theoretically exact with respect to $\lambda$ for every $\alpha,\beta\in [0,1)$.
\end{example}

\begin{figure}[!ht]
	\centering
	\begin{tikzpicture}[scale=3]
	\draw (1,0)--(0,0)--(0,1);
	\draw[dashed] (1,0)--(1,1)--(0,1);
	\draw[thick] (0,1)--(1/5,0)--(2/5,1);
	\draw[thick] (2/5,0)--(3/5,1);
	\draw[thick] (3/5,0)--(4/5,1)--(1,0);
	\node at (2/5,3/5) {$g$};
	\node at (0,-0.1) {$0$};
	\node at (-0.1,1) {$1$};
	\node at (1/5,-0.1) {$\frac{1}{5}$};
	\node at (2/5,-0.1) {$\frac{2}{5}$};
	\node at (3/5,-0.1) {$\frac{3}{5}$};
	\node at (4/5,-0.1) {$\frac{4}{5}$};
	\node at (1,-0.1) {$1$};
	\draw[dotted] (1/5,0)--(1/5,1);
    \draw[dotted] (2/5,0)--(2/5,1);
	\draw[dotted] (3/5,0)--(3/5,1);
    \draw[dotted] (4/5,0)--(4/5,1);
	\end{tikzpicture}
	\caption{Lifting of the map $g\pmod 1$ from Example~\ref{ex:1}.}
	\label{fig:ex1}
\end{figure}
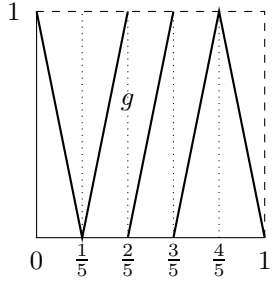

\section{Topological vs. measure-theoretic dynamical properties}

First we show that transitivity and ergodicity are equivalent under some natural assumptions  using the tools developed by Li and Yorke \cite{LY}.
We recall

\ppp*

\begin{proof} Throughout the proof, we use standard definitions from   \cite{LY} without explicitly redefining them. 

Let $F$ denote a lifting of $f$. We define $\Fh:=F|_{[0,1)}\pmod 1$ and call $\Fh$ a {\em representative} of $F$.
Since $f$ is measure-preserving and piecewise $C^2$, for each Borel set $E\subset [0,1]$ the sets $\Fh(E)$ and $\Fh^{-1}(E)$ are also Borel and
\begin{equation}\label{e:0}
\lambda(E)=0\text{ if and only if }\lambda(\Fh(E))=\lambda(\Fh^{-1}(E))=0.
\end{equation}

The fact that ergodicity implies transitivity follows since  each generic point for the measure $\lambda$ on $\Ci$ is transitive.

Let us show the converse. It follows from \cite[Theorem 1]{LY} that
\begin{eqnarray}\label{e:*}
 &\text{the map } \Fh \text{ has exactly one invariant measure absolutely continuous}\\ &\text{ with respect to } \lambda \text{ which is } \lambda \text{ itself.}\nonumber
\end{eqnarray}

 Assume that for some Borel set $A\subset [0,1]$, \begin{equation}\label{e:2}\Fh^{-1}(A)=A,~\lambda\text{-a.e.}\text{ and }\lambda(A)\in (0,1). \end{equation} Write
$$\Fh(A)=\Fh(A\setminus \Fh^{-1}(A))\cup \Fh(A\cap \Fh^{-1}(A)).
$$
But by (\ref{e:2}) and (\ref{e:0}) $\lambda(\Fh(A\setminus \Fh^{-1}(A)))=0$. Moreover $\Fh(A\cap \Fh^{-1}(A))\subset A$ with $\lambda(A\cap \Fh^{-1}(A))=\lambda(A)$ given again by (\ref{e:2}). It shows that $\Fh(A)=A$ $\lambda$-a.e., i.e., $A$ is invariant with respect to $\Fh$. So \cite[Lemma 2.2]{LY} applies: the characteristic function $\chi_A(\cdot)$ is a density of an $\Fh$-invariant measure $\mu$ given for each Borel set $E\subset [0,1]$ by the formula
\begin{equation}\label{e:3}
\mu(E)=\frac{1}{\lambda(A)}\int_{E}\chi_A(x)\di\lambda(x)=\frac{\lambda(E\cap A)}{\lambda(A)}.
\end{equation}
Clearly $\mu$ is an $\Fh$-invariant measure absolutely continuous with respect to $\lambda$ and different from $\lambda$,  contradicting \eqref{e:*}. Thus, in (\ref{e:2}) necessarily  $\lambda(A)\in \{0,1\}$, and thus   $\Fh$, resp.\  $f$ is ergodic.
\end{proof}

We recall and prove

\exact*

\begin{proof}[Proof of Theorem~\ref{exactness}]
 We can view $f$ as a piecewise continuous piecewise $C^2$ interval map. Since   $\lim_{n \to \infty} \lambda(f^n(U))=1$ for any open interval $U$ we get weak mixing of $f$ with respect to $\lambda$ by  \cite[Theorem~2(a)]{Bowen}. 
 By \cite[Theorem~1]{Bowen} the natural extension of $f$ with respect to $\lambda$ (see e.g. \cite[Section 5.10]{NP} for the definition of natural extension in our setting) is weakly Bernoulli (for the definition of weak Bernoulli see \cite{Sh} p. 89) and thus by \cite[Theorem 11.2, 11.3 and 12.1]{Sh} is Bernoulli. This, in particular, implies that $f$ with respect to $\lambda$ is measure-theoretically exact (see \cite{Quas}).
\end{proof}

If $f$ is weakly mixing with respect to a fully supported measure, then it is topologically weak mixing and by Banks and Trotta \cite{BaTr} (cf. Theorem 4.4. from \cite{KKO}) it follows that on topological graphs it is topologically mixing. By Theorem 7.3. \cite{KKO2} a topologically mixing circle map that is not leo needs to have infinitely many fixed points.
 Strictly speaking we have the following.
\begin{remark}\label{rem:leo}
Let $f\in C_{\lambda}(\Ci)$ be weakly mixing w.r.t. $\lambda$. Then $f$ is topologically mixing. If $f$ is additionally piecewise monotone, it is leo.
\end{remark}

\section{Leo is open}

\subsection{Periodic arcs}
We start this subsection with a basic, however very useful lemma.
\begin{lemma}\label{periodic}
Let $f\in C_{\lambda}(\mathbb{S}^1)$ be a piecewise monotone map with slope $>2$ and $k$ critical points. Suppose $f$ has a periodic arc $A=[a_0,b_0]\subset \mathbb{S}^1$ of  {\it minimal} period $n\in\N$.
\begin{itemize}
\item[(1)]For each $i,j\in \Z$, $\lambda(f^{i}(A))=\lambda(f^{j}(A))$.
\item[(2)]\label{2:iii}For $[a_i,b_i]:=f^{i +k n}(A)$ for any $k \ge 0$,
\begin{itemize}
    \item[(i)]  $f(c_i)=a_{i+1}$ and $f(d_i)=b_{i+1}$ for some critical points $c_i,d_i\in  (a_i,b_i)$,
    \item[(ii)]  the points $a_i,b_i$ are not critical and
    \item[(iii)] $f(\{a_i,b_i\})\subset\{a_{i+1},b_{i+1}\}$.
    \end{itemize}
    \item[(3)]\label{periodic:3} $n\le k$.
\end{itemize}
\end{lemma}
\begin{proof}
For any $A \in \A$ we have $A \subset f^{-1}(f(A))$, thus 
since the map $f$ preserves the Lebesgue measure, we have
\begin{equation*}
\lambda(A)\le \lambda(f(A))\le  \cdots \le \lambda(f^{n-1}(A))\le \lambda(f^{n}(A))=\lambda(A),
\end{equation*}
hence (1) follows.
Property (1) easily implies all three properties (2)(i)-(iii). Our assumption (\ref{e:1a}) from the definition of periodic arcs together with the properties (1) and (2)(i) mean that critical points $d_i\in (a_i,b_i)$ satisfying $f(d_i)=b_{i+1}$ are different for different $i$'s, because $b_i\ne b_j$ for $0\leq i< j<n$. 
Therefore, the period $n$ satisfies $n\le k$ as stated in (3).
\end{proof}

The following theorem shows that the phenomenon from Example~\ref{ex:tent} cannot be realized if the slope is strictly greater than $2$.

Let 
$$\begin{array}{lll}
E(f)&:= & \big \{(\alpha,\beta): T_{\alpha,\beta} \text{ has a periodic arc} \big \},\\

E_1(f)&:=& \big \{\alpha: T_{\alpha,0} \text{ has a periodic arc} \big \},\\

E_2(f)&:= & \big \{\beta: T_{0,\beta} \text{ has a periodic arc} \big \}.
\end{array}$$

In practice, the following theorem will be used for piecewise linear maps. For the sake of completeness we state it in a more general form.
\begin{thm}\label{periodn}
Let $f\in C_{\lambda}(\mathbb{S}^1)$ be piecewise monotone and piecewise $C^2$ map with slope strictly greater than $2$. Then
\begin{enumerate}
    \item\label{periodn1}  $E_1(f) = E_2(f)$ are finite sets.
    \item\label{periodn2} $E(f)=  \{(\alpha,0) + (\gamma,-\gamma) \pmod{1}: \alpha \in E_1(f) \text{ and } \gamma \in [0,1)\}$,\\ i.e., $E(f)$ is a finite union of pairwise disjoint topological circles.
    \item\label{periodn3} The following three conditions are equivalent.
    \begin{enumerate}
        \item[(3a)] $T_{\alpha,\beta}$ is leo.
        \item[(3b)] $(\alpha,\beta)\notin E(f)$. 
        \item[(3c)] $T_{\alpha,\beta}$  is exact with respect to $\lambda$.
\end{enumerate}
\end{enumerate}
\end{thm}

\begin{proof}

\eqref{periodn1}   
To see that sets $E_1(f)$ and $E_2(f)$ are of equal cardinality notice that $T_{\alpha,\beta}$ is topologically conjugate through $r_{\gamma}$ to $T_{\alpha+\gamma,\beta-\gamma}$ for any $\gamma$.
In particular for $\alpha=\theta$,
$\beta = 0$ and $\gamma = -\theta$ this yields
 $\theta \in E_2(f)$ provided $\theta\in E_1(f)$. For the converse implication take $\alpha=0$, $\beta=\theta$ and $\gamma=\theta$.

Suppose $\beta \in E_2(f)$.
Since $r_\beta$ is an inner rotation, one can see that the set of critical values of $T_{0,\beta}$ does not depend on a concrete value of $\beta\in [0,1)$, denote this common set by $CV(f)$. Note that by our assumption on the number of critical points of $f$, $\vert CV(f)\vert\le k$. It follows from Lemma \ref{periodic}(2) that  whenever $\beta \in E_2(f)$ has a periodic arc $A_0=[a_0,b_0]\subset \mathbb{S}^1$ of a {\it minimal} period $n\in\N$ then $\bigcup_{i=0}^{n-1}\{a_i,b_i\}\subset CV(f)$ 
and there are maximally $\binom{k}{2}$ of such periodic orbits created by arcs with endpoints in $CV(f)$.

By Lemma \ref{periodic}(2) we additionally have $T_{0,\beta}(\bigcup_{i=0}^{n-1}\{a_i,b_i\}) 
= \bigcup_{i=0}^{n-1}\{a_i,b_i\} \subset CV(f)$.
But for any $x\in \mathbb{S}^1$ there are only finitely many $\beta$ for which $T_{0,\beta}(x)=f(x+\beta)\in CV(f)$, in particular
there are only finitely many $\beta$ for which $a_0$ and $a_0 + \beta$ are in $CV(f)$.
We have shown the desired result,
there can be only finitely many values $\beta\in [0,1)$ such that $T_{0,\beta}$ has a periodic arc.

\eqref{periodn2}  This follows from \eqref{periodn1} and the conjugacy described in the first part of this proof.

 \eqref{periodn3} 
The implication (3a)$\implies$(3b) follows from the definition of leo property. Namely, if $T_{\alpha,\beta}$ is leo then it cannot have a periodic arc.  
(3c)$\implies$(3a) is a consequence of  Remark~\ref{rem:leo}. (3a)$\implies$(3c) follows from Theorem~\ref{exactness}. It remains to prove that (3b)$\implies$(3a).

Thus we only need to show that if $T_{\alpha,\beta}$ is not leo then it has a periodic arc.
 Assume that (A1) $f$ is not leo and does not have a periodic arc.
First we will show that $f$ cannot be transitive; assume the contrary (A2) that $f$ is transitive.
Transitive graph map is either totally transitive or admits periodic decomposition (see Theorem 2.2 in \cite{Alseda}, cf. \cite{Banks}), which in the case of circle means partition into arcs with disjoint interiors, permuted by the map. The latter case is excluded by our assumption, hence the map is totally transitive. But slope $>2$ implies that the map $f$ is not conjugate with irrational rotation which implies it has periodic points (see \cite[Theorem~C]{Alseda2}).
Then $f$ is topologically mixing (see \cite[Theorem~A]{Alseda2}, cf. \cite{KKO,KKO2}). But topologically mixing map which is not leo must have infinitely many fixed points for $f^m$ for some $m$ (see Theorem 7.3 in \cite{KKO2}), which is not the case for piecewise monotone maps. This  is a contradiction of (A2) so $f$ is not transitive. 

Remember that this non-transitive $f$  does not have a periodic arc.
Let $U,V$ be nonempty open subsets of $\mathbb S^1$ such that for each non-negative integer $n$, $f^n(U)\cap V=\emptyset$.  Put
\begin{equation*}
    W:=\overline{\bigcup_{n\ge 0} f^n(U)}
    \end{equation*}
 and suppose $\overline W$ is a connected component of $W.$  By our definition, there exists $j$ fulfilling $f^j(U)\subset \overline W$. Since the map $f$ preserves the Lebesgue measure, $\lambda(U)\le \lambda(f^i(U))$ hence
\begin{equation*}
    \lambda(U)\le \lambda(f^j(U))\le \lambda(\overline W).
\end{equation*}
Thus there are finitely many connected components of $W$, denote them $W_i, 1 \le i \le k$. Using the definition of $W$ and the
pigeonhole principle one can show that
there is $i$ and a number $1\leq s\leq k$ such that $$f^s(W_i)\subset W_i.$$
But $\lambda (W_i)\leq \lambda (f^s(W_i))\leq \lambda(W_i)$, so in fact $W_i=f^s(W_i)$, that is $W_i$ is a periodic arc,  contradicting assumption (A1).
\end{proof}

\subsection{Stably leo piecewise affine maps}
\begin{lem}\label{lem:turning}
Let $f$ be a piecewise affine Lebesgue measure-preserving circle map and $c$ a turning point of $f$. Then at least one of the following two properties holds:
\begin{enumerate}
    \item\label{lem:turning(1)} there exists another turning point $c\neq c'\in \Ci$ such that $f(c)=f(c')$ and if $c$ is a local minimum then $c'$ is a local maximum or vice-versa (see the left-most part of Figure~\ref{fig:7}),
    \item\label{lem:turning(2)} there exists a point $x\in \Ci$ so that $f(x)=f(c)$ and an open neighbourhood $U\ni x$ such that $f|_U$ is injective (see left part of Figure~\ref{fig:8}).
\end{enumerate}
\end{lem}

\begin{figure}[!ht]
	\centering
	\begin{tikzpicture}[scale=2.8]
	\draw (-0.3,0.5)--(1,0.5);
	\draw (0.15,0)--(0.3,0.5)--(0.45,0);
	\draw (-0.1-0.1,0)--(-0.05-0.1,0.5)--(0.1-0.1,0);
	\draw (0.6,0.9)--(0.9,0.1);
	\draw[->] (1.1,0.5)--(1.2,0.5);
	\end{tikzpicture}
	\begin{tikzpicture}[scale=2.8]
	\draw (-0.3,0.5)--(1,0.5);
	\draw (-0.1-0.1,0)--(-0.05-0.1,0.5)--(0.1-0.1,0);
	\draw[dashed] (-0.3,0.9)--(1,0.9);
	\draw[dashed] (-0.3,0.1)--(1,0.1);
	\draw[red, thick] (0.6,0.9)--(0.65,0.5)--(0.7,0.9)--(0.75,0.5);
	\draw[thick] (0.15,0)--(0.3,0.5)--(0.45,0);
	\draw (0.6,0.9)--(0.9,0.1);
	\draw[dashed] (0.6,0.9)--(0.6,0.1);
	\draw[dashed] (0.75,0.9)--(0.75,0.5);
	\draw[thick] (0.75,0.5)--(0.9,0.1);
	\end{tikzpicture}
	\caption{Reduction to Figure~\ref{fig:7}.}\label{fig:8}
	\vspace{0.8cm}
	\begin{tikzpicture}[scale=2.8]
	\draw (-0.1,0.5)--(1,0.5);
	\draw (0.1,1)--(0.3,0.5)--(0.5,1);
	\draw (0.65,0)--(0.8,0.5)--(0.95,0);
	\draw[->] (1.1,0.5)--(1.2,0.5);
	\end{tikzpicture}
	\begin{tikzpicture}[scale=2.8]
	\draw (-0.1,0.5)--(1,0.5);
	\draw[dashed] (-0.1,0.9)--(1,0.9);
	\draw[dashed] (-0.1,0.1)--(1,0.1);
	\draw (0.075,1)--(0.3,0.5)--(0.5,1);
	\draw (0.65,0)--(0.8,0.5)--(0.95,0);
	\draw[->] (1.1,0.5)--(1.2,0.5);
	\draw[dashed] (0.12,0.9)--(0.12,0.5);
	\draw[dashed] (0.3,0.9)--(0.3,0.5);
	\draw[dotted,red,thick] (0.12,0.9)--(0.21,0.5)--(0.26,0.9)--(0.3,0.5);
	\draw[thick] (0.21,0.5)--(0.3,0.5);
	\node at (0.25, 0.4) {\small $\eps$};
	\draw[dashed] (0.68,0.1)--(0.68,0.5);
	\draw[dashed] (0.8,0.1)--(0.8,0.5);
	\draw[thick] (0.71,0.5)--(0.8,0.5);
	\draw[red, thick, dotted] (0.68,0.1)--(0.71,0.5)--(0.75,0.1)--(0.8,0.5);
	\node at (0.75, 0.6) {\small $\eps$};
	\end{tikzpicture}
	\begin{tikzpicture}[scale=2.8]
	\draw (-0.1,0.5)--(1,0.5);
	\draw[dashed] (-0.1,0.9)--(1,0.9);
	\draw[dashed] (-0.1,0.1)--(1,0.1);
	\draw (0.075,1)--(0.3,0.5)--(0.5,1);
	\draw (0.65,0)--(0.8,0.5)--(0.95,0);
	\draw[dashed] (0.12,0.9)--(0.12,0.5);
	\draw[dashed] (0.3,0.9)--(0.3,0.5);
	\draw[dotted,red,thick] (0.12,0.9)--(0.21,0.5)--(0.26,0.9)--(0.3,0.5);
	\draw[thick] (0.21,0.5)--(0.3,0.5);
	\node at (0.25, 0.4) {\small $\eps$};
	\draw[dashed] (0.68,0.1)--(0.68,0.5);
	\draw[dashed] (0.8,0.1)--(0.8,0.5);
	\draw[thick] (0.71,0.5)--(0.8,0.5);
	\draw[red, thick, dotted] (0.68,0.1)--(0.71,0.5)--(0.75,0.1)--(0.8,0.5);
	\node at (0.75, 0.6) {\small $\eps$};
	\draw[red, thick] (0.12,0.9)--(0.21,0.5)--(0.26,0.1)--(0.3,0.5);
	\draw[red, thick] (0.68,0.1)--(0.71,0.5)--(0.75,0.9)--(0.8,0.5);
	\draw[thick] (0.8,0.5)--(0.95,0);
	\draw[thick] (0.3,0.5)--(0.5,1);
	\draw[thick] (0.075,1)--(0.12,0.9);
	\draw[thick] (0.68,0.1)--(0.65,0);
	\end{tikzpicture}
	\caption{Perturbation in case when critical values are different type of extrema.}\label{fig:7}
\end{figure}

\begin{proof}
Assume that we are not in case \ref{lem:turning(1)} and suppose that $f(c)$ is a local maximum. 
Consider a small neighbourhood $U:=(f(c)-\delta,f(c)+\delta)$ of $f(c)$ which contains no other critical value of $f$. Let $U^{-}:=(f(c)-\delta,c]$ and $U^{+}:=[c,f(c)+\delta)$. 
The set $f^{-1}(U)$ consists of a finite number of intervals. Since there is only one critical value in $U$, each such interval needs to be mapped onto either $U$, $U^{-}$ or $U^{+}$.
By the assumption at least one of them gets mapped onto $U^{-}$ and none of them gets mapped only onto $U^{+}$. 
Because $f$ is a Lebesgue measure-preserving circle map at least one of these intervals needs to be mapped onto $U$. The case when $f(c)$ is a local minimum is analogous.
\end{proof}

\begin{lem}\label{lem:13}
Assume that $f\in \mathrm{PA}_{\lambda}(\Ci)$. Arbitrarily close to $f$ we can find a map $h\in \mathrm{PA}_{\lambda}(\Ci)$ all of whose critical values are distinct. Furthermore, if absolute values of slopes of $f$ are bounded below by a constant $s$, then $h$ can be chosen to satisfy this as well. 
\end{lem}

\begin{proof}
Suppose that $c$ is a critical point of $f$, assume without loss of generality that $y := f(c)$ is a local maximum. Let $L_y$ denote the set of turning points for which the  map $f$ has the  critical value $y$.  If $|L_y| = 1$ we have that $f$ has no pairs of turning points with the critical value $y$ which is what we want. 
     
     For $|L_y| > 1$, by Lemma~\ref{lem:turning} if we do not have a pair of critical points as in case \eqref{lem:turning(1)} then all the points in $L$ are local maxima or local minima as in case \eqref{lem:turning(2)}. In the latter case  we perturb $f$ with a regular $3$-fold perturbation (see Figure~\ref{fig:8}),
     which will create one new critical point with critical value $y$ and this value necessarily put us back in case (\ref{lem:turning(1)}) of Lemma~\ref{lem:turning}. The corresponding set $L_y$ is increased by 1 element.
     
     Choose a critical point $c'$ with $f(c')=y$ with local minimum at $c'$. For a small enough $\eps>0$ we first apply a $3$-fold perturbation, as in the dotted red lines middle picture of Figure~\ref{fig:7}. Note that the length of the domain of the two adjacent folds from the $3$-fold perturbations which contain the critical value are required to be of the same length $\eps$. 
     The final perturbation for this pair of critical points is the union of the leftmost fold of the $3$-fold perturbations and the flip of the rest of the $3$-fold perturbations over the critical value axis (as shown with thick lines on the rightmost picture on Figure~\ref{fig:7}). Since these two intervals have the same length $\eps$, the perturbed map preserves Lebesgue measure and the corresponding set $L_y$ is decreased by two elements.  
     
    When we combine the two steps the set $L_y$
has decreased by at least one element (possibly two).
    Repeating these procedures a finite number of times
     produces a map with $|L_y| \le 1$, and then repeating the procedure for other critical values produces the desired map.
     
      Note that the perturbations we perform can be arbitrarily small, and do not decrease the absolute value of the  slopes of the monotone pieces of $f$ and preserve Lebesgue measure. 
\end{proof}

The following statement is proven in \cite[Lemma 11]{BCOT2}.
\begin{lemma}\label{l:6}
Let $F$ be a lifting modulo $1$ of a piecewise affine circle map $f$ with nonzero slopes and such that its derivative does not exist at a finite set $E$. Then $f\in C_{\lambda}(\Ci)$ is equivalent to the property
\begin{equation}
   \forall~y\in [0,1)\setminus F(E)\colon~\sum_{x\in F^{-1}(y)}\frac{1}{\vert F'(x)\vert}=1.
 \end{equation}
 \end{lemma}

Recall that we denote by $\mathcal A$ the set of all nondegenerate proper subarcs of $\mathbb S^1$. For $f\in C_{\lambda}(\mathbb S^1)$ and $A\in\mathcal A$, a nontrivial connected component $A_j$ of $f^{-1}(A)=\bigcup A_j$ satisfies $\lambda(A_j)>0$. It is always an element of $\mathcal A$.

\begin{lemma}\label{lem:leo-exp}
If $h\in \mathrm{PA}_{\lambda}(\mathbb S^1)$ has different critical values and slope strictly greater than $4$ then there are $\delta = \delta(h) >0$ and $\eta = \eta(h) >0 $
such that for every $A\in \mathcal{A}$,
\begin{enumerate}

\item\label{leo-exp1} either $h(A)=\mathbb{S}^1$ or $(1+\delta)\lambda(A)<\lambda(h(A))$. 

\item\label{leo-exp2} for any $A \in \mathcal A$ whose complement contains at most one critical point we have $\lambda(H(I)) > 1+ \eta $ where $H$ is a lifting of $h$ and $\phi(I) = A$. 

\item\label{leo-exp3} for each $(\alpha,\beta)$ the constant  $\delta(T_{\alpha,\beta}(h)) = \delta(h)$.
\item\label{leo-exp4} for each $(\alpha,\beta)$  the constant $\eta(T_{\alpha,\beta}(h)) = \eta(h)$.
\end{enumerate}
\end{lemma}
\begin{proof}
Let $\delta:= \inf\{\frac{\lambda(h(A))}{\lambda(A)} -1: h(A) \ne \mathbb{S}^1\}$. Since $h$ is measure-preserving we have $\delta \ge 0$. There are arcs $A_n\in \mathcal A$ such that $h(A_n)\neq \mathbb{S}^1$ and 
$$\frac{\lambda(h(A_n))}{\lambda(A_n)}\to 1+\delta,$$ as $n\to \infty$.

We may assume that $A':=\lim A_n$ in the Hausdorff distance.
By way of contradiction, suppose that $\delta = 0$; then  $\lambda(h(A')) =\lambda(A')$.
If an interval $A$ contains at most one critical point then $\lambda(h(A))/\lambda(A) >  4/2=2$.  Thus, for all sufficiently large $n$, the interval $A_n$ can not be shorter than the shortest piece of monotonicity of $h$
and thus $\lambda(A') > 0$.

We claim that $\lambda(A') = 1$. Let $h(A') = [c,d]$. Suppose by way of contradiction that  $\lambda(A')<1$. Since the slope of $h$ is strictly greater than $4$ and the critical values of $h$ are different it follows from  Lemma~\ref{l:6} that every point has at least 4 preimages. In particular, the point $y := h(c)$ has at least 4 preimages, at most one of them is a critical point, at most two of them are the endpoints of $A'$. We have shown that there is at least one  preimage in the interior of $A'$ which is not a turning point of $h$.  Thus  there is an arc $[c',c]$ with
$c' < c$ contained in $h(A')$, and so 
$h(A') \supsetneq  [c,d]$, a contradiction. So $\lambda(A') =1$, which finishes the proof of \eqref{leo-exp1}.

 Let $B_n:=\mathbb{S}^1\setminus A_n$.
But then $\lambda (B_n)\to 0$ and so we can choose $n_0$ so that for $n \ge n_0$ the arc $B_n$ contains at most one turning point of $h$. Since
 each point in  $h(B_n)$ has at most two preimages in $B_n$ it must have at least one preimage in $A_n$, in other words $h(B_n) \subset h(A_n)$. But then $h(A_n)\supset h(B_n \cup A_n) = h(\mathbb{S}^1)=\mathbb{S}^1$ which is a contradiction.
 
 \begin{figure}
    \centering
   \begin{tikzpicture}[scale=3]
   \draw(0,0)--(0,1)--(1,1)--(1,0)--(0,0);
   \draw (0.1,0.3)--(0.2,0.8);
   \draw (0.3,0.8)--(0.4,0.3);
   \draw (0.55,0.3)--(0.6,0.55)--(0.65,0.3);
   \draw (0.8,0.3)--(0.9,0.8);
   \draw[dashed] (0,0.55)--(1,0.55);
   \draw[dashed] (0.15,0.55)--(0.15,0);
   \draw[dashed] (0.35,0.55)--(0.35,0);
   \draw[dashed] (0.6,0.55)--(0.6,0);
   \draw[dashed] (0.85,0.55)--(0.85,0);
   \node at (-0.18,0.45) {\small $y-\eps$};
   \node at (-0.18,0.65) {\small $y+\eps$};
   \node at (0.15,-0.1) {\small $A_1$};
   \node at (0.35,-0.1) {\small $A_2$};
   \node at (0.6,-0.1) {\small $B_n$};
   \node at (0.85,-0.1) {\small $A_3$};
   \draw[dotted] (0,0.45)--(1,0.45);
   \draw[dotted] (0,0.65)--(1,0.65);
   \node at (1.2,0.5) {\small $h(B_n)$};
   \node at (0,-0.1) {$0$};
	\node at (1,-0.1) {$1$};
	\node at (-0.1,1) {$1$};
   \draw[ultra thick] (0.125,0)--(0.175,0);
   \draw[ultra thick] (0.325,0)--(0.375,0);
    \draw[ultra thick] (0.825,0)--(0.875,0);
    \draw[ultra thick] (0.575,0)--(0.625,0);
    \draw[dotted](0.575,0)--(0.575,0.45); 
    \draw[dotted](0.625,0)--(0.625,0.45); 
    \draw[ultra thick] (0,0.45)--(0,0.65);
    \draw[ultra thick] (1,0.45)--(1,0.55);
    \end{tikzpicture}
    \caption{Visual aid for Lemmas~\ref{lem:leo-exp}\eqref{leo-exp2} and \ref{lem:compdist}.} 
    \label{twoComponents}
\end{figure}
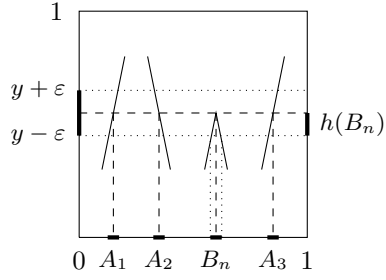
 
 Suppose that $A \in \mathcal A$ and $B = \mathbb{S}^1 \setminus A$ contains at most one critical point.
 Suppose $I = [a,b] \subset J= [a,a+1]$, so $\phi(J) = \mathbb{S}^1$ (recall that $\phi(I) = A$).
 In the previous paragraph we showed that $h(A) = \mathbb{S}^1$. In what follows we denote the critical points of $h$ by $c_i$ and corresponding critical points of $H|_J$ by $C_i$.
 Choose the largest interval $[C_1,C_2] \subset I$ where $C_1,C_2$ are critical points of $H$.  The argument of the previous paragraph shows that if there is a critical point of $h$ in $B$ then $h(c) \in h(A)$ and thus $H(C) \in H(I)$.  All other critical points of $H|_J$ are contained in $[C_1,C_2]$ thus the interval $H([C_1,C_2])$ contains all the critical values of $H|_J$. Furthermore there are distinct critical points $C_3,C_4 \in J$ such that $H([C_1,C_2]) = [H(C_3),H(C_4)]$. In the previous paragraph we showed $h(A) \supset \mathbb{S}^1$, thus $\lambda([H(C_3),H(C_4)])\ge 1$.
 Since $h$ has distinct critical values this inequality is strict.  It is sufficient to choose  $\eta$ to satisfy
 $$1+\eta := \min\{ H(C) - H(C'): |C-C'| < 1, C \ne C' , H(C)-H(C')>1\},$$
 where minimum runs over all critical points $C,C'$ of $H$. This finishes the proof of (\ref{leo-exp2}) and  (\ref{leo-exp4}).

We turn to statement \eqref{leo-exp3}.  Fix $(\alpha,\beta)$
such that $T_{\alpha,\beta}(h)$ is leo.
Consider arcs $A_n$ such that 
$$\frac{\lambda(h(A_n))}{\lambda(A_n)}\to 1+ \delta(h),$$ as $n\to \infty$.  Let
$A_n' := r_{-\beta} A_n$. Then since rotations are measure-preserving we have
$$\frac{\lambda(T_{\alpha,\beta}(h)(A'_n))}{\lambda(A'_n)} = \frac{\lambda(h(A_n))}{\lambda(A_n)}.$$ 
Therefore $\frac{\lambda(T_{\alpha,\beta}(h)(A'_n))}{\lambda(A'_n)} \to 1+\delta(h)$
and so $\delta(T_{\alpha,\beta}(h)) \le \delta(h)$.
Reversing the role of the maps proves the desired equality.
\end{proof}

Lemma~\ref{lem:leo-exp} yields the following corollary.

\corleo*

For $h\in \mathrm{PA}_{\lambda}(\mathbb{S}^1)$ and $A \in \mathcal{A}$ let
$$\tau(A):=\max\{\dist(A',A'')\colon~A',A''\text{ non-degenerate components of }h^{-1}(A)\}.
$$

\begin{lemma}\label{lem:compdist}
Suppose $h\in \mathrm{PA}_{\lambda}(\mathbb S^1)$ is leo with slope strictly greater than $4$ and different critical values. Let $\kappa(h)$ denote the minimal distance between the critical values of $h$. If $A\in\mathcal A$ and $\lambda(A) \le 
{\kappa(h)}$ then the set  $h^{-1}(A)$ has at least two non-degenerate components. 
Moreover there exists $\xi(h) > 0$ such that 
\begin{enumerate}
\item\label{lem:comp1} $\inf_{A\in\mathcal A,~\lambda(A) \le
\kappa} \tau(A)> \xi(h)
$
 \item\label{lem:comp2} for each $A\in\mathcal A$ with $\lambda(A) \le\kappa$ there are arcs $A',A''\subset \mathcal{A}$ such that $\dist(A',A'')\ge \xi(h)$ and $h(A')=h(A'')=A$.
 \item\label{lem:comp3} for each $(\alpha,\beta)$ we have  $\xi(T_{\alpha,\beta}(h)) = \xi(h)$.
  \item\label{lem:comp4} for each $(\alpha,\beta)$ we have  $\kappa(T_{\alpha,\beta}(h)) = \kappa(h)$.
 \end{enumerate}
\end{lemma}
\begin{proof}
Since rotations are isometries, the constant $\kappa(h)$  does not depend on $\alpha$ and $\beta$.
Let $\zeta$ be the length of the shortest piece of monotonicity of $h$.
Since  $h\in \mathrm{PA}_{\lambda}(\mathbb S^1)$ is leo with slope strictly greater than $4$ and different critical values every point $y \in \mathbb{S}^1$ has at least 4 preimages (see Figure \ref{twoComponents}), where at most one of these preimages is a critical point. 
Thus for $0 < \eps \le \kappa/3$  there are at least three disjoint arcs $A_1,A_2,A_3$
such that $h(A_1)=h(A_2)=h(A_3)=(y-\eps,y+\eps)$ and $h$
is monotone on each of these arcs.  In particular two of these arcs $A_i,A_j$ are separated by a piece of monotonicity and so 
 $\dist(A_i,A_j) >  \xi(h)$.
It is easy to see that $\xi(h)\geq \zeta$. Clearly $\zeta$ does not depend on $(\alpha,\beta)$ which finishes the proof of this lemma.
\end{proof}

\subsection{Proof of Theorem~\ref{thm:main}}

Let us recall the statement of the main theorem of this paper.
\main*

\begin{proof}

We start with a set $\{h_n\} \in \mathrm{PA}_{\lambda}(\mathbb{S}^1)$ which is dense in $C_{\lambda}(\mathbb{S}^1)$, e.g. see Lemma 12 from \cite{BCOT2}. 
By applying arbitrary small $n$-fold perturbations we may assume that the slopes of maps $h_n$ are strictly greater than $4$. By Lemma~\ref{lem:13} we may also assume that every $h_n$ has pairwise different critical values. By Corollary~\ref{cor:leo} each $h_n$ is leo and thus also
 Lemma~\ref{lem:leo-exp}  and Lemma~\ref{lem:compdist} hold for $h_n$; we call the constants corresponding to these two lemmas $\kappa_n,\delta_n, \eta_n, \gamma_n$. 
 We can additionally assume that $\eta_n$ is smaller than the length of the shortest piece of  monotonicity of $h_n$ and $\delta_n<1$.

Set $\eps_n<\frac{\delta_n}{2}\min\{\eta_n/6,\kappa_n/3\}$. We will show that the conclusions of the theorem hold for 
$$\mathcal{O}:=\bigcup_{(\alpha,\beta)\in [0,1)^2} T_{\alpha,\beta}\Big(\bigcup_{n\in \N} B(h_n,\eps_n)\Big).$$

Let us prove that any $g\in B(h_n,\eps_n)$ is leo. Recall that $G$ denotes a lifting of $g$ and $H_n$ a lifting of $h_n$.

By Lemma~\ref{lem:leo-exp} and the additional assumption on $\eta_n$ we have
that for every $g\in B(h_n,\eps_n)$:
\begin{itemize}
   \item[(p1)]\label{item1} For an interval $I\subset \mathbb{R}$, if $1-\eta_n<\lambda(I)<1$ then $H_n(I)>1+\eta_n$ and thus $G(I)> 1+\eta_n-2\eps_n>1$ 
   (which implies that $g(A) = \mathbb{S}^1$ for any arc $A\in \mathcal{A}$ satisfying $1-\eta_n < \lambda(A) < 1$).
   \end{itemize}
   
   We claim that additionally
    \begin{itemize}
\item[(p2)] If $A\in\mathcal A$ and $\lambda(A)<\alpha_n:=\min\{\xi_n/2,\kappa_n/3\}$ then $g^{-1}(A)$ has at least two non-degenerate components.
\end{itemize}

 Let us prove (p2). Observe that 
 $\lambda(B(A,\eps_n))=\lambda(A)+2\eps_n<\kappa_n/3+2\eps_n\leq \kappa_n$. So we can apply
 Lemma~\ref{lem:compdist} \eqref{lem:comp2}
to conclude that  $h_n^{-1}(B(A, \eps_n))$ has at least two non-degenerate components $A'$ and $A''$ such that $h_n(A')=h_n(A'')=B(A, \eps_n)$ and $\dist(A',A'')>\xi_n$. 
Thus since $\rho(g,h_n)<\eps_n$
 we have $A \subset g(A')\cap g(A'')$. Thus there exist intervals $Q'\subset A', Q''\subset A''$ so that $g(Q')=g(Q'')=A$ with $\lambda(Q'),\lambda(Q'')>0$. 
Furthermore, since $\dist(A',A'')>\xi_n$ and $Q' \subset A'$ and $Q''\subset A''$ we conclude $\dist(Q',Q'')>\xi_n$ as well, i.e.,
 $Q',Q''$ are distinct non-degenerate components of $g^{-1}(A)$ as required.

Fix $g \in B(h_n,\eps_n)$. If  $g$ is topologically mixing for every arc $A \in \mathcal{A}$ we have $\lambda(g^n(A)) \to 1$. By (p1) we get $g^n(A) = \mathbb{S}^1$ for some sufficiently large $n$.
Thus we have
\begin{align*}g\text{ is not leo}~\implies~g~\text{is not topologically mixing}.
\end{align*}

In what follows we will repeatedly use the following fact: if $Q \in \mathcal{A}$
then $Q \subset g^{-1}(g(Q))$, thus since $g$ is measure-preserving we have $\lambda(Q) \le \lambda(g(Q))$.

If $g$ is not topologically mixing then by definition there are arcs
$A,\tilde A\in\mathcal A$ such that $g^{n_k}(A)\cap \tilde A=\emptyset$ for an  infinite sequence $\{n_k\}_{k\in \N}$. 
In particular,
$\lambda(g^{n_k} (A)) \le 1-\lambda(\tilde A)$ for each  $k$. By the above fact, the sequence
$\{\lambda(g^{n_k}(A))\}_{n=0}^{\infty}$
is non-decreasing, thus since  $\mathbb S^1$ is
compact, there exists an arc $B\in\mathcal A$ such that  for some
subsequence $\{{n_k}_i\}_{i\in \N}$,
\begin{equation}
\label{en:1}
\lim_{i\to\infty}g^{{n_k}_i}(A)=B
\end{equation}
satisfying $\lambda(A) \le \lambda(B) \le 1 - \lambda(\tilde A) < 1$.

Suppose that $a := \lambda(g(B))- \lambda(B) > 0$, then since $g$ is continuous it follows that there exists a $K\in \N$ such that  $\lambda(g^{{n_k}_i +1}(A))-\lambda(B) > a/2$ for every $i \ge K$. 
Again applying the above fact yields $\lambda(g^{n}(A)) - \lambda(B) > a/2$ for every $n \ge n_K$. This holds in particular for all ${n_k}_i$ for $i$ sufficiently large which contradicts \eqref{en:1}. Thus we have shown that
\begin{equation}\label{en:2}
    \lambda(g(B)) = \lambda(B) \in (0,1)
\end{equation}

This implies that the set $g^{-1}(g(B))$ has only one non-degenerate component; hence by (p1) and (p2), \begin{equation}\label{bounds}\alpha_n\leq  \lambda(B)  \leq 1-\eta_n.
\end{equation}

Note that by definition $\eps_n<\frac{\delta_n}{2}$,
thus combining \eqref{en:2}, the fact that $g \in B(h_n, \eps_n)$, Lemma~\ref{lem:leo-exp}\eqref{leo-exp1} (which we can apply because of the upper bound in \eqref{bounds}), and the lower estimate of \eqref{bounds} yields
$$
\lambda(B)=\lambda(g(B))\geq \lambda(h_n(B)) - 2\eps_n \geq(1+\delta_n)\lambda(B)  - 2\eps_n \geq \lambda(B)+\delta_n\alpha_n - 2 \eps_n>\lambda(B).
$$
This contradiction finishes the proof of (1).

Statement \eqref{thm:main2} follows directly from the definition of $\mathcal{O}$. Namely, for the proof of \eqref{thm:main2} it suffices to recall that $T_{\alpha,\beta}$ is an isometry and thus $g \in T_{\alpha_k,\beta_k}(B(f_{n_k},\varepsilon_{n_k}))$ if and only if $T_{\alpha,\beta}(g)\in T_{\alpha+\alpha_k,\beta+\beta_k}(B(f_{n_k},\varepsilon_{n_k}))$.
\end{proof}

\section{Generic weak-mixing}

 In \cite[Theorem 2]{BCOT} we showed that there is a dense set of interval maps in $C_{\lambda}([0,1])$ which are not ergodic, the proof works with natural modifications in  $C_{\lambda}(\mathbb{S}^1)$, thus a strong metric result in the spirit of Theorem \ref{thm:main} can not hold. Nevertheless we can still obtain the following theorem.

\tfive*

\begin{proof}
Using Theorem~\ref{exactness} and Corollary~\ref{cor:leo}, if $h\in \mathrm{PA}_{\lambda}(\mathbb S^1)$ has different critical values and slope strictly greater than $4$ then it is exact with respect to $\lambda$ (and thus also strongly and weakly mixing with respect to $\lambda$).
Thus we have
 a dense set of maps  $\{f_n: n\in \N\} \subset C_{\lambda}(\mathbb{S}^1)$  such that $T_{\alpha,\beta}(f_n)$ is weakly mixing with respect to $\lambda$ for each $\alpha,\beta\in [0,1)$.
Let $\{h_j\}_{j \ge 1}$ be a countable, dense
collection  of continuous functions in $L^1(\mathbb{S}^1 \times \mathbb{S}^1)$.
For any $f \in C_{\lambda}(\mathbb{S}^1)$ and $\ell \ge 1$, let
$$S^f_{\ell}h_j(x,y)  := \frac{1}{\ell} \sum_{k=0}^{\ell-1} h_j \big( (f \times f)^k(x,y)\big).$$
The map $f$ is weakly mixing if and only if the map $f \times f$ is ergodic, and
by the Birkhoff ergodic theorem, the map $f \times f$ is ergodic  if and only if  we have
$$\lim_{\ell \to \infty} S^f_\ell h_j(x,y)  = \int_{\mathbb{S}^1 \times \mathbb{S}^1} h_j(s,t) \, d(\lambda(s) \times \lambda (t))$$
for all $j \ge 1$.

For each $n \ge 1$ and $(\alpha,\beta)$ there exists a set $E_{n,\alpha,\beta} \subset \mathbb{S}^1 \times \mathbb{S}^1$  and a positive integer $\ell_{n,\alpha,\beta} \ge n$ such that
$\lambda (E_{n,\alpha,\beta}) > 1 - \frac1n$ and
\begin{equation}\label{est''}\Big |S^{T_{\alpha,\beta}(f_n)}_{\ell} h_j(x,y) - \int_{\mathbb{S}^1 \times \mathbb{S}^1} h_j(s,t) \,  d(\lambda(s) \times \lambda (t)) \Big | < \frac{1}{3n}
\end{equation}
for all $(x,y) \in E_{n,\alpha,\beta}$, $1 \le j \le n$, and $\ell \ge \ell_{n,\alpha,\beta}$.

For any $\alpha,\beta$, any $\bar \alpha,\bar\beta$ and any $g\in C_{\lambda}(\Ci)$ by the triangular inequality we have:
$$\begin{array}{ll}
 \Big |S^{T_{\alpha,\beta}(g)}_{\ell_{n, \bar \alpha, \bar \beta}} h_j(x,y)   -  \int_{\mathbb{S}^1 \times \mathbb{S}^1} h_j(s,t) \,  d(\lambda(s) \times \lambda (t)) \Big |    \le\\
 \qquad \Big |S^{T_{\alpha,\beta}(g)}_{\ell_{n, \bar \alpha, \bar \beta}} h_j(x,y)  - S^{T_{\alpha,\beta}(f_n)}_{\ell_{n, \bar \alpha, \bar \beta}} h_j(x,y) \Big | + \\
  \qquad \quad \Big |S^{T_{\alpha,\beta}(f_n)}_{\ell_{n, \bar \alpha, \bar \beta}} h_j(x,y)  - S^{T_{\bar\alpha,\bar\beta}(f_n)}_{\ell_{n, \bar \alpha, \bar \beta}} h_j(x,y) \Big |  +  \\
  \qquad \quad\quad  \Big |S^{T_{\bar\alpha,\bar\beta}(f_n)}_{\ell_{n, \bar \alpha, \bar \beta}} h_j(x,y)  -   \int_{\mathbb{S}^1 \times \mathbb{S}^1} h_j(s,t) \,  d(\lambda(s) \times
 \lambda (t)) \Big |
 .\end{array}$$
 
 We would like to show that the latter is small when $g$ is sufficiently close to $f_n$ and $\alpha,\beta$ is sufficiently close to $\bar\alpha,\bar\beta$ and $(x,y)\in E_{n,\bar\alpha,\bar\beta}$.
 
 Remember that $h_j$'s are continuous functions.
 Since $T_{\alpha,\beta}(g)$ is continuous in both $(\alpha,\beta)$ and $g$, the first term is bounded by $\frac{1}{3n}$ for a small neighborhood of $(\bar \alpha, \bar \beta, g)$.
 Since $T_{\alpha,\beta}$ is continuous in $(\alpha,\beta)$ the second term is bounded by $\frac{1}{3n}$ for a small neighborhood of $(\bar \alpha, \bar\beta )$ in $[0,1)^2$.
 Thus we can choose a small neighborhood 
 $I_{\bar \alpha, \bar \beta}$ of 
$(\bar \alpha, \bar \beta)$ in $[0,1)^2$
and a small neighborhood $B(f_n,\eps_{n,\bar \alpha,\bar \beta})$ in $C_{\lambda}(\Ci)$ such that both of the above estimates hold.
The third term is bounded by $\frac{1}{3n}$ by \eqref{est''}, therefore we get
 for any point $(x,y) \in E_{n,\bar \alpha, \bar \beta}$, any $(\alpha,\beta) \in I_{\bar \alpha,\bar \beta}$,   
 and $g \in B(f_{n},\varepsilon_{n,\bar \alpha, \bar\beta})$ that
\begin{equation}\label{here''}\Big |S^{T_{\alpha,\beta}(g)}_{\ell_{n,\bar \alpha,\bar \beta}} h_j(x,y)  -  \int_{\mathbb{S}^1 \times \mathbb{S}^1} h_j(s,t) \,  d(\lambda(s) \times \lambda (t)) \Big |    < \frac1n
\end{equation}
for $1 \le j\le n$.

Since the torus $[0,1)^2$ is compact
we can find a finite cover of it by such sets
$I_{n,\bar\alpha_{n,i},\bar\beta_{n,i}}, 1 \le i \le j_{n}$.
Choose $\eps_n := \min \{\eps_{n,\bar\alpha_{n,i},\bar\beta_{n,i}}: 1 \le i \le j_n\}$.

Consider
 the set
 $$\mathcal G:=\bigcap_{N \ge 1}
  \bigcup_{(\alpha,\beta) \in [0,1)^2 }T_{\alpha,\beta}\Big( \bigcup_{n \ge N}  B(f_n,\varepsilon_n)\Big).$$
  
 The set $\mathcal G$ is a dense $G_\delta$ set since $T_{\alpha,\beta}$ is an isometry. 
  
For each  $g \in \mathcal G$
there are infinite sequences $n_k$ and $(\alpha_k,\beta_k)$ and 
$(\bar \alpha_{n_k,i_k},\bar \beta_{n_k,i_k})$
such that $g \in T_{\alpha_k,\beta_k}(B((f_{n_k}),\varepsilon_{n_k}))$ and
$(\alpha_k,\beta_k) \in I_{n_k,\bar \alpha_{n_k,i_k},\bar \beta_{n_k,i_k}}$.
Consider
$$\mathcal{E}(g)  := \bigcap_{M=1}^\infty \bigcup_{k=M}^\infty  E_{n_k,\bar \alpha_{n_k,i_k},\bar \beta_{n_k,i_k}}.$$
Since $\lambda(E_{n_k,\bar \alpha_{n_k,i_k},\bar \beta_{n_k,i_k}}) >
1 - \frac{1}{n_{k}}$, it follows that  $\lambda(\mathcal{E}(g)) = 1$.

We have shown that 

\begin{equation}\label{13}\Big| S^g_{\ell_{n_k},\bar \alpha_{n_k,i_k},\bar \beta_{n_k,i_k}} h_j(x,y)  -  \int_{\mathbb{S}^1 \times \mathbb{S}^1} h_j(s,t) \,  d(\lambda(s) \times \lambda (t)) \Big |    < \frac{1}{n_k}
\end{equation}
for all $(x,y) \in E_{n_k,\bar \alpha_{n_k,i_k},\bar \beta_{n_k,i_k}}$ and all $1 \le j\le n_k$ for each $k \ge 1$.

The
$
\lim_{\ell \to \infty} S^{g}_{\ell} h_j(x,y)
$
exists for almost every $(x,y)$ by the Birkhoff ergodic theorem. Equation~\eqref{13} holds  infinitely often for every point of $\mathcal{E}(g)$  so we conclude that this limit must equal
 $\int_{\mathbb{S}^1 \times \mathbb{S}^1} h_j(s,t) \,  d(\lambda(s) \times \lambda (t))$ and thus
 $g \times g$ is ergodic, or equivalently $g$ is weakly mixing.

For the proof of (2) is suffices to note that
 $g \in T_{\alpha_k,\beta_k}(B((f_{n_k}),\varepsilon_{n_k}))$ if and only if $T_{\alpha,\beta}(g)\in T_{\alpha+\alpha_k,\beta+\beta_k}(B((f_{n_k}),\varepsilon_{n_k}))$.
\end{proof}

 \section*{Acknowledgements}
We would like to thank CIRM for hospitality during our research in residence program where this paper was finalized.\\
J. Bobok was supported by the European Regional Development Fund, project No.~CZ 02.1.01/0.0/0.0/16\_019/0000778.
J. \v Cin\v c was supported  by the IDUB program no. 1484 ``Excellence initiative – research university'' for the AGH University of Science and Technology. P. Oprocha was supported by National Science Centre, Poland (NCN), grant no. 2019/35/B/ST1/02239.

\end{document}